\documentclass[runningheads]{llncs}
\usepackage[breaklinks, colorlinks, citecolor=blue]{hyperref}
\usepackage{color}

\newcommand{\ourparagraph}[1]{\medskip \noindent\textit{#1}.}

\usepackage{amsmath}
\usepackage{amssymb}
\usepackage{bbm}
\usepackage{textgreek}
\usepackage{mathtools}
\usepackage{calc}
\usepackage{xspace}
\usepackage{xcolor}
\usepackage{booktabs}
\usepackage{multirow}
\usepackage{collcell}
\usepackage{tikz}
\usepackage{tikzscale}
\usepackage{pgfplots}
\pgfplotsset{compat=1.9}

\newcommand{\eg}{e.g.,\xspace}

\pgfplotsset{
	/pgf/number format/textnumber/.style={
		fixed,
		fixed zerofill,
		precision=2,
		1000 sep={.},
	},
}
\newcommand{\ApplyColor}[1]{\ifdim #1pt<0.95pt \textcolor{blue}{#1}\else
	\ifdim #1pt>1.05pt \textcolor{red}{#1}\else #1 \fi
	\fi
}

\newcolumntype{R}{>{\collectcell\ApplyColor}{r}<{\endcollectcell}}

\newcommand{\R}{\mathbb{R}\xspace}

\newcommand{\Bins}{V\xspace}
\newcommand{\NetFlow}{f}
\newcommand{\Coherence}{g}
\newcommand{\Cluster}{\mathcal{K}\xspace}

\newcommand{\clustind}{\kappa}
\newcommand{\Edges}{E\xspace}
\newcommand{\Trans}{A\xspace}

\newcommand{\coh}{\Coherence\xspace}
\newcommand{\flow}{\NetFlow\xspace}

\newcommand{\percentage}[2]{\pgfmathparse{(1 - #1/#2) *100}\pgfmathprintnumber[precision=1]{\pgfmathresult}\%}
\newcommand{\reduction}[2]{\pgfmathparse{(#1)/#2 }{\pgfmathprintnumber[textnumber]{\pgfmathresult}}}

\newcommand{\colorquot}[2]{
	\pgfmathparse{(#1<0.9*#2)?1:0}\ifdim\pgfmathresult pt>0pt \textcolor{blue}{$\mathbf{\reduction{#1}{#2}}$}\else
	\pgfmathparse{(#1>1.1*#2)?1:0}\ifdim\pgfmathresult pt>0pt \textcolor{red}{$\mathbf{\reduction{#1}{#2}}$} \else 
	$\reduction{#1}{#2}$\fi \fi
}

\usepackage{algorithm}
\usepackage{algpseudocode}

\DeclareMathOperator{\conv}{conv}

\usepackage{enumitem}
\usepackage{array}
\newcolumntype{$}{>{\global\let\currentrowstyle\relax}}
\newcolumntype{^}{>{\currentrowstyle}}

\begin{document}
\title{Branch and Cut for Partitioning a Graph into a Cycle of Clusters}
\author{Leon Eifler\inst{1}\orcidID{0000-0003-0245-9344} \and
Jakob Witzig\inst{1,2}\orcidID{0000-0003-2698-0767} \and
Ambros Gleixner\inst{1,3}\orcidID{0000-0003-0391-5903}}
\authorrunning{L. Eifler, J. Witzig, A. Gleixner}
\institute{Zuse Institute Berlin, Takustr. 7, 14195 Berlin, Germany \email{eifler@zib.de} \and
SAP SE, Dietmar-Hopp-Allee 17, 69190 Walldorf, Germany \email{jakob.witzig@sap.com} \and
HTW Berlin, 10313 Berlin, Germany \email{gleixner@htw-berlin.de}}
\maketitle              \begin{abstract}
  In this paper we study formulations and algorithms for the cycle clustering problem, a
  partitioning problem over the vertex set of a directed graph with nonnegative arc weights
  that is used to identify cyclic behavior in simulation data generated from
  nonreversible Markov state models.
Here, in addition to partitioning the vertices into a set of coherent clusters, the
  resulting clusters must be ordered into a cycle such as to maximize the total net flow
  in the forward direction of the cycle.
We provide a problem-specific binary programming formulation and compare it to a
  formulation based on the reformulation-linearization technique (RLT).
We present theoretical results on the polytope associated with our custom formulation and
  develop primal heuristics and separation routines for both formulations.
In computational experiments on simulation data from biology we find that branch and cut
  based on the problem-specific formulation outperforms the one based on RLT.

\end{abstract}
\section{Introduction}

This paper is concerned with solving a curious combination of a graph partitioning problem
with a cyclic ordering problem called \emph{cycle clustering}.
This combinatorial optimization problem was introduced in~\cite{WitzigEtal2018}
to analyze simulation data from nonreversible stochastic processes, in particular
Markov state models, where standard methods using spectral information are not easily
applicable.
One typical example for such processes are catalytic cycles in biochemistry, where a set
of chemical reactions transforms educts to products in the presence of a catalyst.

Generally, we are given a complete undirected graph $G=(\Bins, \Edges)$ with vertex set
$\Bins =\{1,\ldots, n\}$, edge set $\Edges = \binom{n}{2}$,
weights~$q_{i,j} \ge 0$ for all~$i,j\in\Bins$,
and a pre-specified number of clusters~$m\in\mathbb{N}$.
It is essential to the problem that $q$ is not assumed to be symmetric, i.e., we may have
$q_{i,j}\not=q_{j,i}$.
We refer to directed edges $(i,j)$ also as arcs and let $A=\{(i,j) \in \Bins\times\Bins :
i\not=j\}$.
Our goal is then to partition $\Bins$ into an ordered set of clusters
$C_1,\dots, C_m$ such as to maximize an objective function based
on the following definition.
\begin{definition}
  Let $Q\in\R_{\geq 0}^{n\times n}$ and $S,T \subseteq \Bins=\{1,\ldots, n\}$ be two
  disjoint sets of vertices, then we define the \emph{net flow} from $S$ to $T$ as
\begin{align*}
    \flow(S,T) \coloneqq \sum_{i \in S,j\in T} q^-_{i,j}
    &\,\text{ with }\, q^-_{i,j} \coloneqq q_{i,j} - q_{j,i},
    \intertext{and the \emph{coherence} of $S$ as}
    \coh(S) \coloneqq \sum_{i,j \in S} q_{i,j} = \sum_{i,j \in S, i\leq j} q^+_{i,j}
    &\,\text{ with }\, q^+_{i,j} \coloneqq \begin{cases}
      q_{i,i}, & \text{if } i=j, \\
      q_{i,j} + q_{j,i}, & \text{otherwise}.
    \end{cases}
  \end{align*}
\end{definition}
Given a scaling parameter $\alpha$ with $0 < \alpha < 1$, we want to compute a
partitioning $C_1,\dots, C_m$ of $\Bins$ that maximizes the combined objective function
\begin{align}
  \label{eq:ccobj}
  \alpha \sum_{t=1}^m \flow(C_{t},C_{t+1}) + (1-\alpha) \sum_{t=1}^m \coh(C_t),
\end{align}
where we use the cyclic notation $C_{m+1} = C_{1}, C_{m+2} = C_2, \ldots,$ when indexing
clusters in order to improve readability.
W.l.o.g. we assume in the following that $q_{i,i}=0$ for all~$i\in\Bins$, since nonzero values only add a
constant offset to~\eqref{eq:ccobj}.

In the application domain, the vertices may correspond to Markov states and $Q =
(q_{i,j})_{i,j\in\Bins}$ then constitutes the transition matrix of the unconditional
probabilities that a transition from state~$i$ to state~$j$ is observed during the
simulation.
In this case, the coherence of a set of states gives the probability that a transition takes
place inside this set, hence the second component of~\eqref{eq:ccobj} aims to separate
states that rarely interact with each other and cluster states into so-called
metastabilities.
The (forward) flow $\sum_{i \in S,j\in T} q_{i,j}$ between two disjoint sets of states~$S$
and $T$ corresponds to the probability of observing a transition from $S$ to $T$.
Hence, the first component of~\eqref{eq:ccobj} aims at arranging clusters into a cycle
such that the probability of observing a transition from one cluster to the next minus the
probability of observing a transition in the backwards direction is maximized.
With $\alpha$ close to one, finding near-optimal cycle clusterings can be used to detect
whether there exists such cyclic behavior in simulation data.
For details on the motivation and interpretation of cycle clustering, we refer
to~\cite{WitzigEtal2018}.

Since the solution of cycle clustering problems to global optimality has proven to be
challenging, we have developed a problem-specific branch-and-cut algorithm that is the
focus of this paper.
Our contributions are as follows.
In Sec.~\ref{sec:formulations}, we first present two binary programming formulations for cycle clustering and give basic results on the underlying polytopes.
In Sec.~\ref{sec:inequalities}, we provide the basis for a branch-and-cut algorithm by
introducing three classes of valid inequalities and discuss suitable separation algorithms
for them;
we show that the computationally most useful class of extended subtour and path
inequalities can be separated in polynomial time.
In Sec.~\ref{sec:heuristics}, we describe four primal heuristics that aim at quickly
constructing near-optimal solutions before or during branch and cut: a greedy construction
heuristic, an LP-based rounding heuristic, a Lin-Kernighan~\cite{KernighanLin1970} style
improvement heuristic, and a sub-MIP heuristic based on sparsification of the arc set.
We conclude the paper with Sec.~\ref{sec:experiments}, where we evaluate and compare the
effectiveness of the primal and dual techniques numerically on simulation data that are
generated to exhibit cyclic structures.
 \section{Formulations}
\label{sec:formulations}

We start with a straightforward formulation as a quadratic binary program.
Let $\Cluster \coloneqq \{1,\ldots,m\}$ denote the index set of clusters.
For each vertex $i \in \Bins $ and cluster $C_s$, $s \in \Cluster$,
we introduce a binary variable $x_{i,s}$ with
\begin{align*}
  x_{i,s} = 1 &\,\Leftrightarrow\, \text{vertex } i \text{ is assigned to cluster } C_s. 
\end{align*}
Then a first, nonlinear formulation for the cycle clustering problem is
\begin{subequations}
\begin{align}
  \max\hspace*{0.9em}\alpha\;
  & \lefteqn{\sum_{s \in \Cluster} \sum_{(i,j) \in \Trans} q^-_{i,j}\, x_{i,s} x_{j,s+1}
    + (1-\alpha) \sum_{s \in \Cluster} \sum_{\{i,j\} \in \Edges} q^+_{i,j}\, x_{i,s} x_{j,s}} \\
  \text{s.t.} \hspace*{2.4em}
  &\sum_{s \in \Cluster} x_{i,s} = 1 && \text{for all } i \in \Bins, \label{eq:assignment} \\
  &\sum_{i \in \Bins} x_{i,s} \ge 1 && \text{for all } s \in \Cluster, \label{eq:setcover} \\
  &x_{i,s} \in \{0,1\} && \text{for all } i \in \Bins, s \in \Cluster. \hspace*{7em}
\end{align}
\end{subequations}
Again, note that we use the cyclic notation $x_{i,m+1} = x_{i,1}$ for convenience.
The two parts of the objective function correspond to net flow and coherence, respectively.
Constraints~\eqref{eq:assignment} ensure that each vertex is assigned to exactly one cluster, while constraints~\eqref{eq:setcover} ensure that no cluster is empty.

Following Padberg~\cite{Padberg1989}, this binary quadratic program could be linearized by
introducing a binary variable and linearization constraints for each bilinear term, also
called McCormick inequalities~\cite{mccormick1976computability}.
A more compact linearization proposed by Liberti~\cite{Liberti2007} and refined by Mallach~\cite{Mallach2018} exploits the set partitioning constraints~\eqref{eq:assignment}.
Instead of McCormick inequalities, it relies on the reformulation-linearization technique (RLT) \cite{SheraliAdams1999}.
Each of equations in \eqref{eq:assignment} is multiplied by $x_{j,t}$ for all $j \in \Bins$, $t \in \Cluster$, and the bilinear terms~$x_{i,s}x_{j,t}$ are replaced by newly introduced linearization variables $w_{i,j}^{s,t}$.
This gives the mixed-binary program~(RLT)
\begin{align*}
  \max\hspace*{0.9em}\alpha\;
  & \lefteqn{\sum_{s \in \Cluster} \sum_{(i,j) \in \Trans} q^-_{i,j}\, w_{i,j}^{s,s+1}
    + (1-\alpha) \sum_{s \in \Cluster} \sum_{\{i,j\} \in \Edges} q^+_{i,j}\, w_{i,j}^{s,s}} \\
  \text{s.t.} \hspace*{2.4em}
  &\eqref{eq:assignment}, \eqref{eq:setcover}, \\
  &\sum_{s \in \Cluster} w_{i,j}^{s,t} = x_{j,t} && \text{for all } i,j \in \Bins, t \in \Cluster, \\
  &x_{i,s} \in \{0,1\} && \text{for all } i \in \Bins, s \in \Cluster, \hspace*{7em} \\
  &w_{i,j}^{s,t} \in [0,1] && \text{for all } (i,j) \in \Trans, s,t \in \Cluster,
  \intertext{where, by symmetry arguments, we can use the same variable for $w_{i,j}^{s,s}=w_{j,i}^{s,s}$.}
\end{align*}

\pagebreak

Our second, problem-specific formulation exploits the fact that we only need to
distinguish three cases for each edge $\{i,j\} \in \Edges$ in order to model its
contribution to the objective function.
Either $i$ and $j$ are in the same cluster, in consecutive clusters, or more than one
cluster apart.
We capture this by introducing binary variables $y_{i,j}$ with
\begin{align*}
  y_{i,j} = 1 &\,\Leftrightarrow\, \text{vertices } i \text{ and } j \text{ are in the same cluster}
  \intertext{for all edges $\{i,j\} \in \Edges$, $i < j$, and}
  z_{i,j} = 1 &\,\Leftrightarrow\, \text{vertex } i \text{ is one cluster before } j \text{ along the cycle}
\end{align*}
for all arcs $(i,j)\in\Trans$.
We will use the shorthand $y_{i,j}=y_{j,i}$ if $i > j$.
With these variables, we obtain the significantly more compact binary program (CC)
\begin{subequations}
\begin{align}
  \max\hspace*{1ex}
  \alpha\! \sum_{(i,j) \in \Trans} q^-_{i,j}\, z_{i,j}
  + (1-\alpha)\!\! \sum_{\{i,j\} \in \Edges} q^+_{i,j}\, y_{i,j}
  \hspace*{1.1em} \\
  \text{s.t.} \hspace*{14.8em}
  \eqref{eq:assignment}, \eqref{eq:setcover}&,
  \notag \\
  y_{i,j} + z_{i,j} + z_{j,i} \le 1 &  \quad\text{for all } \{i,j\} \in \Edges, \label{eq:oneedge}\\
   x_{i,s} + x_{j,s} - y_{i,j} + z_{i,j} - x_{j,s+1} - x_{i,s-1} \le 1 & \quad\text{for all } (i,j) \in \Trans, s \in \Cluster, \label{eq:incluster}\\
   x_{i,s} + x_{j,s+1} - z_{i,j} + y_{i,j} - x_{j,s} - x_{i,s+1} \le 1 & \quad\text{for all } (i,j) \in \Trans, s \in \Cluster, \label{eq:concurrcluster}\\
  x_{i,s} \in \{0,1\} & \quad\text{for all } i \in \Bins, s \in \Cluster, \\
  y_{i,j},z_{i,j}, z_{j,i}  \in \{0,1\} & \quad\text{for all } \{i,j\} \in \Edges, i < j. \label{eq:ccpend}
\end{align}
\end{subequations}
Again, the first sum in the objective function expresses the total net flow between
consecutive clusters, while the second sum expresses the coherence within all clusters.
Constraints $\eqref{eq:oneedge}$ ensure that two vertices cannot both be in the same
cluster and in consecutive clusters.
Constraints \eqref{eq:incluster} and \eqref{eq:concurrcluster} are best explained by
examining several weaker constraints first.
For $(i,j) \in \Trans$, $s \in \Cluster$, consider \begin{subequations}
\begin{align}
  x_{i,s} + x_{j,s} - y_{i,j} &\le 1 \label{weak1},\\
  x_{i,s} + z_{i,j} - x_{j,s+1} &\le 1 \label{weak2},\\
  x_{j,s} + z_{i,j} - x_{i,s-1} &\le 1 \label{weak3}.
\end{align}
\end{subequations}
The reasoning behind \eqref{weak1} is that if $i$ and $j$ are in the same cluster, then $y_{i,j}$ has to be equal to one. If $i$ is in some cluster $s$ and $z_{i,j}=1$, then \eqref{weak2} forces $j$ to be in the next cluster $s+1$. Analogously, if $j$ is in cluster $s$ and $z_{i,j} = 1$, then \eqref{weak3} ensures that $i$ is in the cluster preceding $s$. 
All of these three cases are covered by \eqref{eq:incluster}, since $y_{i,j}$ and $z_{i,j}$ are binary and at most one of the two can be nonzero at the same time. 
Constraints \eqref{eq:concurrcluster}, in the same way, cover the functionality of the weaker constraints
\begin{align*}
  x_{i,s} + x_{j,s+1} - z_{i,j} &\le 1, \\
  x_{i,s} + y_{i,j} - x_{j,s} &\le 1, \\
  x_{j,s+1} + y_{i,j} - x_{i,s+1} &\le 1.
\end{align*}
Since \eqref{eq:incluster} and \eqref{eq:concurrcluster} are defined for all $(i,j) \in \Trans$ and $s \in \Cluster$, the following implications hold:
\begin{itemize}
\item If $i$ and $j$ are in the same cluster, then $y_{i,j} = 1$ due to \eqref{eq:incluster}.
\item If $i$ and $j$ are in consecutive clusters, then $z_{i,j} = 1$ due to \eqref{eq:concurrcluster}.
\item If $y_{i,j} = 1$, there exists $s \in \{1,\ldots,m\}$ with ${x_{i,s}= x_{j,s} = 1}$ due to \eqref{eq:concurrcluster}.
\item If $z_{i,j} = 1$, there exists $s \in \{1,\ldots,m\}$ with ${x_{i,s}= x_{j,s+1} = 1}$ due to \eqref{eq:incluster}.
\end{itemize}

To compare both formulations (RLT) and (CC), we first note that one can transform any
feasible solution of (RLT) into a feasible solution of (CC) and vice versa, since the
assignment of the $x$-variables in both cases uniquely implies the assignment of the
remaining linearization variables.
Moreover, it holds for all $(i,j) \in \Trans, s \in \Cluster$ that 
$y_{i,j} = \sum_{s \in \Cluster} w_{i,j}^{s,s}$
and
$z_{i,j} = \sum_{s \in \Cluster} w_{i,j}^{s,s+1}$,
hence also any valid inequality for (CC) can be transformed into a valid inequality for
(RLT).

Second, we observe that the LP relaxations of both formulations yield the same dual bound
at the root node.
For both formulations an LP-optimal solution is induced by assigning $x_{i,s} = 1 / m$
for all $i \in \Bins$, $s \in \Cluster$.

Regarding size, (RLT) has roughly $0.5m^2$ times as many variables as (CC), whereas
(CC) has about twice the number of constraints.
Note that in practice we can exploit in both formulations that
linearization variables for pairs of vertices $i,j\in\Bins$ are only needed if
$q_{i,j}+q_{j,i}>0$.

For the following section, define the \emph{cycle clustering polytope} $CCP$ as the convex
hull of all feasible incidence vectors of (CC), i.e.,
\begin{equation*}
  CCP \coloneqq \conv\big( \{ (x,y,z) \in \mathbb{R}^{nm+1.5|\Trans|} : \text{(\ref{eq:assignment}--\ref{eq:setcover}), (\ref{eq:oneedge}--\ref{eq:ccpend})} \} \big)
\end{equation*}
One can show that the dimension of $CCP$ is $(m-1)n + |\Trans|$ if $m=3$ and $(m-1)n +
1.5|\Trans|$ if $4\leq m \leq n-2$.
Furthermore, the lower bound constraints $x_{i,s} \ge 0$, $y_{i,j} \ge 0$, and $z_{i,j} \ge
0$ define facets of $CCP$, whereas the upper bound constraints $x_{i,s} \leq 1$, $y_{i,j}
\leq 1$, and $z_{i,j} \leq 1$ do not.
For $m \geq 4$, one can show that also the model inequalities \eqref{eq:incluster} and
\eqref{eq:concurrcluster} define facets of $CCP$.
 \section{Valid Inequalities}
\label{sec:inequalities}

As explained in Sec.~\ref{sec:formulations}, any valid inequality for (CC) can be
transformed to a valid inequality for (RLT).
Hence, the following valid inequalities for the cycle clustering polytope $CCP$ apply
equally to (RLT).
The formal proofs that some of the inequalities are facet-defining follow standard
patterns and are left for an extended version of the paper due to space limitations.

In this section we use the notation $\clustind(i)\in\{1,\ldots,m\}$ to denote the (index
of the) cluster that a vertex~$i\in\Bins$ is assigned to.

\subsection{Triangle Inequalities}

First, consider the special case of exactly three clusters $m=3$ and three vertices~$i,j,k\in \Bins$.
If $j$ is in the successor cluster of $\clustind(i)$, 
and $k$ is in the successor cluster of $\clustind(j)$, 
then also $i$ has to be in the successor cluster of~$\clustind(k)$.
This is expressed by the valid inequality
\begin{align*}
  z_{i,j}+z_{j,k} - z_{k,i} \le 1 && \text{for all } (i,j),(j,k),(k,i) \in \Trans.
\end{align*}
Now assume $m \ge 4$.
The next inequality builds only on $y$-variables and can be found in the graph
partitioning literature, \eg~\cite{ChopraRao1993,Grotschel1990}.
If $i$ and $j$, and $j$ and $k$ are in the same cluster, then also
$i$ and $k$ must be in the same cluster, i.e., we have
\begin{align}
  \label{eq:trianlge-y}
  y_{i,j} + y_{j,k} - y_{i,k} \le 1 &&  \text{for all } (i,j),(j,k),(k,i) \in \Trans.
\end{align}
Similarly, if $i$ and $j$ are in the same cluster, and $k$ is in the successor of $\clustind(i)$,
then $k$ also has to be in the successor of $\clustind(j)$.
This gives the valid inequality
\begin{align}
  \label{eq:triangle-yz-1}
  y_{i,j} + z_{i,k} - z_{j,k} \le 1 &&  \text{for all } (i,j),(j,k),(k,i) \in \Trans,
\end{align}
and if we choose $k$ to be in the predecessor of $\clustind(i)$, then we obtain
\begin{align}
  \label{eq:triangle-yz-2}
  y_{i,j} + z_{k,i} - z_{k,j} \le 1 &&  \text{for all } (i,j),(j,k),(k,i) \in \Trans.
\end{align}
With these observations, one can prove the following result.
\begin{theorem}
  Let $m\geq 4$ and $i,j,k \in \Bins$ with $(i,j),(j,k),(i,k) \in \Trans$, then
\begin{equation}
    \label{ineq:triangle-y}
    y_{i,j}+y_{j,k}-y_{i,k} + \frac{1}{2}\left(z_{i,j}+z_{j,i}+z_{j,k} + z_{k,j} - z_{i,k} - z_{k,i}\right) \le 1
  \end{equation}
is a valid, facet-defining inequality for $CCP$.
\end{theorem} 

The final set of triangle inequalities is derived from the following observations.
If vertex~$i$ is in the predecessor of~$\clustind(j)$ as well as in the predecessor of~$\clustind(k)$, then
$j$ and $k$ must be assigned to the same cluster.
This gives the valid inequality
\begin{align}
  \label{ineq:facet-triangle-forward}
  z_{i,j} + z_{i,k} - y_{j,k} \le 1 && \text{for all } (i,j),(i,k),(j,k) \in \Trans.
\end{align}
Conversely, if $i$ is in the successor of~$\clustind(j)$ as well as in the successor of~$\clustind(k)$,
then $j$ and $k$ must be assigned to the same cluster, i.e., we have
\begin{align}
  \label{ineq:facet-triangle-backward}
  z_{j,i} + z_{k,i} - y_{j,k} \le 1 && \text{for all } (i,j),(i,k),(j,k) \in \Trans.
\end{align}
In the special case that $m = 4$ we can prove that the stronger inequality
\begin{equation}
  \label{ineq:facet-triangle-4}
  z_{i,j} + z_{i,k} - 2y_{j,k} - (z_{j,k} + z_{k,j} + z_{j,i} + z_{k,i}) \le 0
\end{equation}
must hold.	
The reason is that if $j$ is assigned to the successor of~$\clustind(i)$, but $k$ is not assigned
to the successor of~$\clustind(i)$, then $k$ has to be in one of the other three clusters.
In each of those three cases, one of the variables $z_{j,k}$, $z_{k,j}$, $z_{k,i}$ has to be
set to one:
If $k$ is in the same cluster as $i$, then $z_{k,j}$ has to be one;
if $k$ is in the predecessor of $\clustind(i)$, then $z_{k,i}$ has to be one;
if $k$ is in the successor of $\clustind(j)$, then $z_{j,k}$ has to be one.
The same argument holds if $z_{i,k}$ is one, but $z_{i,j}$ is not.
Regarding the strength of the above triangle inequalities, we can show the following.

\begin{theorem}
  Let $i,j,k \in \Bins$ with $(i,j),(j,k),(i,k) \in \Trans$. 
If $m=4$, then \eqref{ineq:facet-triangle-4} is facet-defining for $CCP$.
If $m > 4$, then \eqref{ineq:facet-triangle-forward} and \eqref{ineq:facet-triangle-backward} are facet-defining for $CCP$.
\end{theorem}
All types of triangle inequalities can be separated at once by complete enumeration in $\mathcal{O}(n^3)$~time.

\subsection{Partition Inequalities}
	
The following generalization of the triangle
inequalities~\eqref{ineq:facet-triangle-forward} is inspired by~\cite{Grotschel1990}.
\begin{theorem}
  Let $S,T \subseteq \Bins$ with $S \cap T = \emptyset$. The \emph{partition inequality}
\begin{equation}
    \label{ineq:partition}
    \sum_{i \in S, j \in T} z_{i,j} - \sum_{i,j \in S,\, i < j} y_{i,j} - \sum_{i,j \in T,\, i < j} y_{i,j} \le \min \{|S|,|T|\}
  \end{equation}
is valid for $CCP$. If $m > 4$ and $|S|-|T|=\pm 1$, it is facet-defining.
\end{theorem}
These inequalities can be separated heuristically by deriving $S,T$ from almost
violated triangle inequalities.
To limit the computational effort, we only create partition inequalities with $|S|+|T|\le
5$ in our implementation.

\subsection{Subtour and Path Inequalitities}

Since we consider a fixed number of clusters, if $(i,j)$ is an arc between two clusters,
then there must be exactly $m - 1$ further arcs along the cycle before vertex~$i$ is
reached again.
Formally, let $K = \{(i_1,i_2),(i_2,i_3),\ldots,(i_{\ell-1},i_{\ell}), (i_\ell,i_1)\}\subseteq E$
be any cycle of length $1 < |K| < m$, then the \emph{subtour (elimination) inequality}
\begin{equation}
  \label{ineq:subtour-easy}
  \sum_{(i,j) \in K} z_{i,j} \le |K| - 1 
\end{equation}
is valid for $CCP$.
This inequality can be extended by adding variables for arcs inside a cluster.
Let $U \subset K$ be any strict subset of $K$, then the \emph{extended subtour
(elimination) inequality}
\begin{equation}
  \label{ineq:subtour}
  \sum_{(i,j) \in K} z_{i,j} + \sum_{(i,j)\in U} y_{i,j} \le |K| - 1
\end{equation}
is valid for $CCP$, because in any cycle the number of forward transitions must be a
multiple of the number of clusters~$m$.
The inequality \eqref{ineq:subtour} is stronger than \eqref{ineq:subtour-easy} in the
sense that it defines a higher-dimensional face of $CCP$.
Figure~\ref{fig:subtourpath} illustrates an extended subtour inequality.
While extended subtour inequalities must not define facets, they prove effective in practice and can be separated efficiently.

\begin{theorem}
  Extended subtour inequalities can be separated in polynomial time.
\end{theorem}
\begin{proof}
  Let $(x,y,z)$ be a given LP solution.
By rotational symmetry, we may assume w.l.o.g. that $U = K \setminus \{(i_1,i_2)\}$.
For a fixed start node $i_1$, define weights
\begin{align*}
c_{i,j} \coloneqq \begin{cases}
      z_{i,j}, &\text{ if } i = i_1, \\
      z_{i,j} + y_{i,j}, &\text{ otherwise,}
    \end{cases}
  \end{align*}
for all $(i,j)\in\Trans$, and let $A' = \{(i,j)\in\Trans : c_{i,j} > 0\}$.
Then violated extended subtour inequalities correspond to cycles in the
  directed graph $D=(\Bins, A')$ with length $\ell < m$ and weight greater than $\ell - 1$.

  Hence, \eqref{ineq:subtour} can be separated by computing, for each start node, maximum
  weight walks between all pairs of nodes for $\ell = 2,...,m - 1$.
Since $c_{i,j}\geq 0$, this is possible in $\mathcal{O}(n^3m)$ time by dynamic programming.
\end{proof}

Similarly,
let $P = \{(i_1,i_2,),\ldots, (i_{m-1},i_m)\}$ be a path from $i_1$ to $i_m\neq i_1$ of
length $m-1$, and let $U \subset P$.
Then the \emph{path inequality}
\begin{equation}
  \sum_{(i,j) \in P}z_{i,j}+ \sum_{(i,j)\in U} y_{i,j} + y_{i_1,i_m} \le m - 1
  \label{ineq:path}
\end{equation}
is valid for $CCP$, for the following reason.
Suppose $\sum_{(i,j) \in P}z_{i,j}+ \sum_{(i,j)\in U} y_{i,j}$ equals $m-1$, then there are between one
and $m-1$ forward transitions from $i_1$ to $i_m$. Therefore, $i_1$ and $i_m$ cannot be in
the same cluster and $y_{i_1,i_m}$ has to be zero.
These inequalities can be separated in the same way as the extended subtour inequalities.
Figure~\ref{fig:subtourpath} gives an example of a path inequality.

\begin{figure}[t]
  \centering
  \resizebox{0.35\linewidth}{!}{
	\begin{tikzpicture}[
	statenode/.style = {draw,shape=circle,fill=blue, inner sep=0pt, minimum size = 2pt},
	]
	\usetikzlibrary{calc}
	
	\node (B) at (360/5: 3cm)[draw, circle , minimum size = 3cm]{};
	\node (A) at (2*360/5: 3cm)[draw, circle , minimum size = 3cm]{};
	\node (E) at (3*360/5: 3cm)[draw, circle , minimum size = 3cm]{};
	\node (D) at (4*360/5: 3cm)[draw, circle, minimum size = 3cm]{};
	\node (C) at (360: 3cm)[draw, circle , minimum size = 3cm]{};

	\node (i_1) at (2*360/5: 3cm) [shift = {(-1,-0.8)},draw,shape=circle,fill=blue, inner sep=0pt, minimum size = 4pt] {};
	\node (i_2) at (2*360/5: 3cm)[shift = {(0.2,0.2)},draw,shape=circle,fill=blue, inner sep=0pt, minimum size = 4pt] {};
	\node (i_3) at (360/5: 3cm)[shift = {(0,1)},draw,shape=circle,fill=blue, inner sep=0pt, minimum size = 4pt] {};
	\node (i_4) at (360/5: 3cm)[shift = {(-0.2,-0.6)}, draw,shape=circle,fill=blue, inner sep=0pt, minimum size = 4pt] {};
	\node (i_5) at (360: 3cm)[shift = {(0.2,0.5)},draw,shape=circle,fill=blue, inner sep=0pt, minimum size = 4pt] {};

	\draw[->, color = red, line width = 2pt] (i_1) to node[midway, shift={(0.5,-0.1)}]{\Large $z_{i_1,i_2}$}(i_2){};
	\draw[<->, color = blue, line width = 2pt, bend left] (i_1) to node[midway, shift={(-0.5,0.1)}]{\Large $y_{i_1,i_2}$}(i_2){};
	\draw[->, color = red, line width = 2pt] (i_2) to node[midway, above left]{\Large $z_{i_2,i_3}$}(i_3){};
	\draw[->, color = red, line width = 2pt] (i_3) to node[midway, below left]{\Large $z_{i_3,i_4}$}(i_4){};
	\draw[<->, color = blue, line width = 2pt, bend left] (i_3) to node[midway, right]{\Large $y_{i_3,i_4}$}(i_4){};
	\draw[->, color = red, line width = 2pt] (i_4) to node[midway, above right]{\Large $z_{i_4,i_5}$}(i_5){};
	\draw[->, color = red, line width = 2pt] (i_5) to node[midway, below]{\Large $z_{i_5,i_1}$}(i_1){};
	
	\end{tikzpicture}
}
\quad\qquad
\resizebox{0.35\linewidth}{!}{
	\begin{tikzpicture}[
	statenode/.style = {draw,shape=circle,fill=blue, inner sep=0pt, minimum size = 2pt},
	]
	\usetikzlibrary{calc}
	
	\node (B) at (360/5: 3cm)[draw, circle , minimum size = 3cm]{};
	\node (A) at (2*360/5: 3cm)[draw, circle , minimum size = 3cm]{};
	\node (E) at (3*360/5: 3cm)[draw, circle , minimum size = 3cm]{};
	\node (D) at (4*360/5: 3cm)[draw, circle, minimum size = 3cm]{};
	\node (C) at (360: 3cm)[draw, circle , minimum size = 3cm]{};

	\node (i_1) at (3*360/5: 3cm) [shift = {(-0.3,-0.2)},draw,shape=circle, inner sep=0pt,fill=blue, minimum size = 4pt, label={[xshift=0cm,yshift=-.7cm]}] {};
	\node (i_2) at (2*360/5: 3cm)[shift = {(0.2,0.2)},draw,shape=circle, inner sep=0pt,fill=blue, minimum size = 4pt] {};
	\node (i_3) at (360/5: 3cm)[shift = {(.5,.7)},draw,shape=circle, inner sep=0pt, fill=blue, minimum size = 4pt] {};
	\node (i_5) at (360: 3cm)[shift = {(0.2,0.5)},draw,shape=circle, inner sep=0pt,fill=blue, minimum size = 4pt,label={[xshift=0cm,yshift=-.8cm]}] {};
	\node (i_4) at (360/5: 3cm)[shift = {(.5,-.5)},draw,shape=circle,inner sep=0pt,fill=blue,minimum size=4pt]{};
	
	\draw[->, color = red, line width = 2pt] (i_1) to node[midway, shift={(0.5,-0.1)}]{\Large $z_{i_1,i_2}$}(i_2){};
	\draw[->, color = red, line width = 2pt] (i_2) to node[midway, above left]{\Large $z_{i_2,i_3}$}(i_3){};
	\draw[->, color = red, line width = 2pt] (i_3) to node[midway, below left]{\Large $z_{i_3,i_4}$}(i_4){};
	\draw[<->, color = blue, line width = 2pt, bend left] (i_3) to node[midway, right]{\Large $y_{i_3,i_4}$}(i_4){};
	\draw[->, color = red, line width = 2pt] (i_4) to node[midway, below left]{\Large $z_{i_4,i_5}$}(i_5){};
	
	\draw[<->, color = blue, line width = 2pt, bend left] (i_5) to node(prob)[midway, below right]{\Large $y_{i_5,i_1}$}(i_1){};
\end{tikzpicture}
}
   \caption{Illustration of an extended subtour inequality (left) and a path inequality (right) in a $5$-cluster problem.  Not all of the displayed variables can be set to one.}
  \label{fig:subtourpath}
\end{figure}
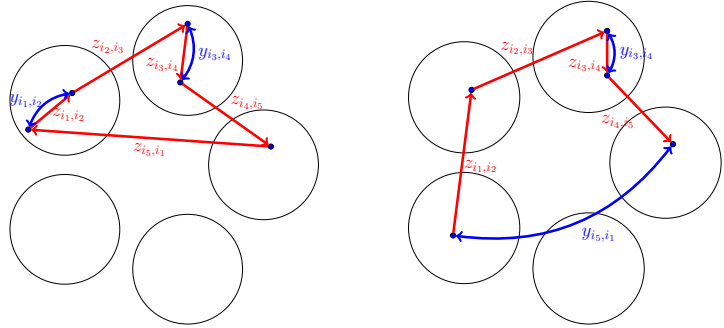
 \section{Primal Heuristics}
\label{sec:heuristics}

The heuristics presented in this section exploit the fact that an integer solution for
(RLT) and (CC) is determined completely by the assignment of $x$-variables.
We present three start heuristics and one improvement heuristic.

\ourparagraph{Greedy}
This heuristic constructs a feasible solution by iteratively assigning vertices to clusters in a greedy fashion.
First, one vertex is assigned to each cluster, such that \eqref{eq:setcover} is satisfied.
Second, we repeatedly compute for all unassigned vertices and each of the $m$~clusters
the updated objective value (in $\mathcal{O}(n)$~time), and choose the
vertex-cluster assignment with largest improvement.
In our implementation, the algorithm is called once as a start heuristic.
It runs in $\mathcal{O}(n^2m)$~time and always generates a feasible solution since it ensures that each cluster contains at least one vertex and that all vertices are assigned.

\ourparagraph{Sparsify}
For this sub-MIP heuristic, we remove 97\% of the edges with smallest weight
$q_{i,j}+q_{j,i}$ and call the branch-and-cut solver on the reduced instance
with a node limit of one, i.e., we only process the root node in order to limit its effort.

\ourparagraph{Rounding}
This heuristic assumes a fractional LP solution and rounds the $x$-variables.
For each vertex~$i\in \Bins$, it selects a cluster $s^*\in\Cluster$ with largest $x_{i,s^*}$
and fixes it to one; all other $x_{i,s}$, $s\not=s^*$, are fixed to zero.
If $\arg\max_{s^*\in\Cluster} x_{i,s^*}$ is not unique, we choose the smallest cluster.
For each vertex, $m$~clusters have to be considered, hence the heuristic runs in
$\mathcal{O}(nm)$~time.
If no cluster remains empty, the heuristic terminates with a feasible solution.

\ourparagraph{Exchange}
This heuristic assumes an integer feasible solution and applies a series of exchange steps, where each time one vertex is transferred to a different cluster.
It is inspired by the famous Lin-Kernighan graph partitioning
heuristic~\cite{KernighanLin1970}.
Each vertex is reassigned once.
We repeatedly compute for all unprocessed vertices and each of the $m-1$~alternative
clusters the updated objective value (in $\mathcal{O}(n)$~time), and choose the
swap with largest improvement, even if it is negative.
We track the solution quality and return the clustering with best objective value.

The algorithm runs in $\mathcal{O}(n^3m)$~time and is applied to each new
incumbent.
It can be executed repeatedly, with the solution of the previous run as the
input, until no more improvement is found.
In this case, it can be useful to continue from a significantly different solution.
To this end, half of the vertices are selected randomly from each cluster and moved to the next cluster.
In our implementation, we apply at most 5~perturbations.
 \section{Computational Experiments}
\label{sec:experiments}

\newcommand{\rlt}{RLT\xspace}
\newcommand{\rltsep}{RLT+sepa\xspace}
\newcommand{\ccmip}{CC\xspace}
\newcommand{\ccmipsep}{CC+sepa\xspace}

\def \timerlt         {1184.1} 
\def \timerltsep      {1377.9} 
\def \timeccmip       {1550.9} 
\def \timeccmipsep    {755.3} 
\def \nodesrlt        {182.0} 
\def \nodesrltsep     {35.3} 
\def \nodesccmip      {676.4} 
\def \nodesccmipsep   {125.0} 
\def \gaprlt          {72.3} 
\def \gaprltsep       {69.1} 
\def \gapccmip        {35.9} 
\def \gapccmipsep     {25.2} 
\def \intprimrlt      {2540.7}
\def \intprimrltsep   {2552.5}
\def \intprimccmip    {1161.7}
\def \intprimccmipsep {1077.5}
\def \intdualrlt      {47568.0}
\def \intdualrltsep   {39205.7}
\def \intdualccmip    {81561.4}
\def \intdualccmipsep {34755.5}

\def \timedef    {1592.6}
\def \timepart   {1372.0}
\def \timesubt   {949.9}
\def \timetri    {1147.4}
\def \timesepa   {876.0}
\def \nodesdef   {1240.6}
\def \nodespart  {593.4}
\def \nodessubt  {408.1}
\def \nodestri   {639.0}
\def \nodessepa  {227.0}

\def \onetimerlt         {221.8}
\def \onetimerltsep      { 299.0}
\def \onetimeccmip       {376.8}
\def \onetimeccmipsep    { 89.0}
\def \onenodesrlt        {206.1}
\def \onenodesrltsep     {45.6}
\def \onenodesccmip      {1425.2}
\def \onenodesccmipsep   {137.6}
\def \onegaprlt          {3.2}
\def \onegaprltsep       {2.1}
\def \onegapccmip        {5.3 }
\def \onegapccmipsep     {0.7}
\def \oneintprimrlt      {622.9}
\def \oneintprimrltsep   {692.4}
\def \oneintprimccmip    {245.2}
\def \oneintprimccmipsep {221.4}
\def \oneintdualrlt      {6305.5}
\def \oneintdualrltsep   {4037.6}
\def \oneintdualccmip    {10445.9}
\def \oneintdualccmipsep {1311.3}
 
We implemented both formulations (RLT) and (CC) within the branch-and-cut framework SCIP~\cite{GleixnerEtal2018} and added the separation routines and primal heuristics described above.
Except for the problem-specific changes, we used default settings for all test runs.
The time limit was set to $7200$~seconds, with only one job running per node at a time on a compute cluster of 2.7 GHz Intel Xeon E5-2670 v2 CPUs with 64 GB main memory each.

We consider $65$~instances over graphs with $20$ to $250$~vertices that were created from
four different types of simulations of non-reversible Markov state models such as to
exhibit between $3$ and $15$~clusters.
In the objective function, net flow and coherence are weighted $1:0.001$, i.e., $\alpha =
1/1.001$.
The first $40$~instances stem from a model of artificial catalytic cycles,
simulated using a hybrid Monte-Carlo method~\cite{brooks2011handbook} as described
in~\cite{WitzigEtal2018}, with $3$, $4$, and $6$~minima;
for each of those, transition matrices over graphs with $20$, $30$, $50$, and
$100$~vertices were created.
For $12$~instances we used repressilator simulations~\cite{repress}, a
very prominent example of a synthetic genetic regulatory network;
these instances feature $40, 80$, and $200$~vertices and between $3$ and
$6$~clusters.
Another $8$~instances were created from simulations of the Hindmarsh-Rose
model~\cite{hindrose}, which is used to study neuronal activity in the human heart;
these instances feature $50$ and $250$~vertices with $3$, $5$, $7$, and $9$~clusters each.
The final set of $5$~instances stems from the dynamics of an
Amyloid-\textbeta\ peptide~\cite{ReuterWeberFackeldeyetal.2018}
over $220$ vertices, with $3$ to $15$~clusters.

We first compare the performance of (RLT) and (CC), both with and without problem-specific separation.
In all cases, the problem-specific heuristics were enabled, since they apply equally to both models.
Table~\ref{tbl:formulations} reports aggregated results, using shifted geometric means as
defined in~\cite{Achterberg2007} with a shift of $10$~seconds for time, $100$~nodes for
the size of the search tree, and $1000$ for the primal and dual integrals.
The primal integral~\cite{Berthold2013} is a useful metric to
quantify how quickly good primal solutions are found and improved during the course of the
solving process.
The dual integral is defined analogously and can give give insight into how quickly good
dual bounds are established.
It is computed using the difference between the current dual bound and the best known primal solution.
The gap at primal bound $p$ and dual bound $d$ is computed as $(d - p) / \min\{ p +
\epsilon, d + \epsilon\}$ for $\epsilon=10^{-6}$ and can exceed 100\%~\cite{Achterberg2007}.

\begin{table}[t]
\centering\small
\caption{Performance of formulations \ccmip and \rlt with and without problem-specific
  separation: number of solved instances, shifted geometric means of solving time, number
  of branch-and-bound nodes, primal and dual integral, and arithmetic mean of final
  gap. Subset ``1-solved'' contains all instances that could be solved to optimality by at
  least one setting.
Problem-specific heuristics are enabled, hence the lines ``\ccmip'' and ``\ccmipsep''
  over all instances match the last two lines in Tab.~\ref{tbl:sepaheur}.}
\label{tbl:formulations}
\begin{tabular*}{\textwidth}{l@{\extracolsep{\fill}}lrrrrrr}
\toprule
Instances & setting & solved & time [s]       & nodes     &  gap [\%] & primal int. &    dual int. \\
\midrule
All (65)
&\rlt        &       29 &     \timerlt        &  \nodesrlt     & \gaprlt     & \intprimrlt      &   \intdualrlt \\
&\rltsep     &       28 &     \timerltsep     &  \nodesrltsep  & \gaprltsep  & \intprimrltsep   &   \intdualrltsep \\
&\ccmip      &       24 &     \timeccmip      &  \nodesccmip   & \gapccmip   & \intprimccmip    &   \intdualccmip \\
&\ccmipsep   &       31 &     \timeccmipsep   &  \nodesccmipsep& \gapccmipsep& \intprimccmipsep &   \intdualccmipsep \\
\midrule
1-solved (34)
& \rlt       &       29 &     \onetimerlt     &  \onenodesrlt     & \onegaprlt     & \oneintprimrlt      &   \oneintdualrlt \\
& \rltsep    &       28 &     \onetimerltsep  &  \onenodesrltsep  & \onegaprltsep  & \oneintprimrltsep   &   \oneintdualrltsep \\
& \ccmip     &       24 &     \onetimeccmip   &  \onenodesccmip   & \onegapccmip   & \oneintprimccmip    &   \oneintdualccmip \\
& \ccmipsep  &       31 &     \onetimeccmipsep&  \onenodesccmipsep& \onegapccmipsep& \oneintprimccmipsep &   \oneintdualccmipsep \\
\bottomrule
\end{tabular*}
\end{table}
 
While plain (CC) solves less instances than both variants of (RLT), we observe that (CC)
with separation clearly outperforms all other settings overall.
(CC) with separation solves $7$~more instances than plain (CC) and reduces solving time by
\percentage{\timeccmipsep}{\timeccmip} and the number of nodes by
\percentage{\nodesccmipsep}{\nodesccmip}.
By contrast, for (RLT) we observe that separation negatively impacts overall performance.
While it reduces the number of nodes significantly by a similar factor than in (CC), the
addition of cutting planes slows down the speed of LP solving and overall node processing
time.
This may be explained by the fact that (RLT) contains many more variables than
constraints.
The above analysis is confirmed on the subset ``1-solved'' of instances, where all
instances not solved by any setting are excluded.

\begin{table}[t]
\centering\small
\caption{Performance impact of separation routines and primal heuristics over all
  65~instances with formulation \ccmip: on the number of solved instances, on the shifted
  geometric means of solving time, number of branch-and-bound nodes, primal and dual
  integral, and on the arithmetic mean of the final gap.}
\label{tbl:sepaheur}
\begin{tabular*}{\textwidth}{l@{\extracolsep{\fill}}rrrrrr}
\toprule
Setting
& solved & time [s] & nodes &    gap [\%] & primal int. & dual int. \\
\midrule
Default        &       27 &     \timedef  & \nodesdef   & 7439.0 & 62250.3 & 76759.2 \\
\quad +triangle only      &       29 &     \timetri  & \nodestri   & 7231.0 & 42281.0 & 48884.2 \\
\quad +partition only     &       27 &     \timepart & \nodespart  & 7652.5 & 58147.5 & 60835.6 \\
\quad +subtour only       &       30 &     \timesubt & \nodessubt   & 4906.8 & 45601.8 & 45380.6 \\
\quad +all sepas     &       31 &     \timesepa & \nodessepa  & 5050.1  & 40796.6 & 38811.6 \\
\midrule
Default        &       27 &     \timedef  & \nodesdef   & 7439.0 & 62250.3 & 76759.2 \\
\quad +greedy only        &       25 &     1741.3 &         1084.3 & 1356.1  &   49072.8 &        72946.5 \\
\quad +sparsify only       &       23 &     1763.2 &         1109.6 &  113.5  &   47536.5 &        75186.5 \\
\quad +rounding only      &       27 &     1376.4 &          937.4 & 1724.3  &   34584.9 &        50266.7 \\
\quad +exchange only      &       28 &     1305.2 &          684.4 &   35.0  &   1252.5  &        73321.3 \\
\quad +all heurs     &       24 &     1550.9 &          676.4 &   35.9  &   1161.7  &        81561.4 \\
\midrule
All sepas+all heurs &  31 &     \timeccmipsep& \nodesccmipsep& \gapccmipsep& \intprimccmipsep &   \intdualccmipsep \\
\bottomrule
\end{tabular*}
\end{table}
 
In order to quantify the performance impact of individual solving techniques, we continued
to use (CC) as a baseline with all cutting planes and primal heuristics deactivated.
In Tab.~\ref{tbl:sepaheur} we report the performance when a single separator is enabled,
when all separators are enabled simultaneously, and similarly for the four primal
heuristics.
Here we can observe that each separation routine individually improves performance.
Subtour separation seems to be the most useful single separator, reducing the solving time
by \percentage{\timesubt}{\timedef} and the number of nodes by
\percentage{\nodessubt}{\nodesdef}.
All separators combined deliver the best results, with $4$ more instances solved, a
speedup of \percentage{\timesepa}{\timedef}, and a reduction in the number of nodes by
\percentage{\nodessepa}{\nodesdef}.
However, the final gap still remains large because on hard instances that time out, good
primal solutions cannot be found by the LP solutions alone.

Among the primal heuristics, the exchange heuristic proves to be the most important single
technique.
While also the other heuristics all reduce the primal integral and the final gap, only the
exchange heuristic increases the number of solved instances (by one instance) and yields
the smallest final gap of 35.0\%.
Still, applying all heuristics together gives the best results.
The main benefit of the greedy and sparsify heuristics is to run once in order to provide
the exchange heuristic with starting solutions to improve upon.
The best performance in all metrics is obtained when all separators and heuristics are
active.

To conclude, we observed the best out-of-the-box performance using an RLT formulation
including improvements from~\cite{Liberti2007,Mallach2018}.
However, this formulation did not profit from problem-specific separation due to an
increased cost in solving the LP relaxations.
In turn, our problem-specific formulation responded very well to the separation routines,
and we observed that all solving techniques complement each other in improving
performance.
Still, only \percentage{31}{65} of the instances could be solved to proven optimality, which
underlines the importance of the primal heuristics when applying cycle clustering to
real-world data.
In this respect, we found that repeated application of a Lin-Kernighan style exchange
heuristic with random perturbations was highly successful in improving initial solutions
provided by simple construction heuristics.

 \begin{credits}
\subsubsection{\ackname}
We wish to thank Konstantin Fackeldey, Andreas Grever, and Marcus Weber for supplying us
with simulation data for our experiments.
This work has been supported
by the Research Campus MODAL
funded
by the German Federal Ministry of Education and Research (BMBF
grants 05M14ZAM, 05M20ZBM).
\subsubsection{\discintname}
The authors have no competing interests to declare that are relevant
to the content of this article.
\end{credits}

\clearpage
\bibliographystyle{splncs04}
\bibliography{Bibliography}

\begin{thebibliography}{10}
\providecommand{\url}[1]{\texttt{#1}}
\providecommand{\urlprefix}{URL }
\providecommand{\doi}[1]{https://doi.org/#1}

\bibitem{Achterberg2007}
Achterberg, T.: Constraint Integer Programming. Ph.D. thesis (2007),
  \url{http://nbn-resolving.de/urn/resolver.pl?urn:nbn:de:0297-zib-11129}

\bibitem{SheraliAdams1999}
Adams, W.P., Sherali, H.D.: A reformulation-linearization technique for solving
  discrete and continuous nonconvex problems (1999).
  \doi{10.1007/978-1-4757-4388-3}

\bibitem{Berthold2013}
Berthold, T.: Measuring the impact of primal heuristics. OR Letters
  \textbf{41}(6),  611--614 (2013).
  \doi{http://dx.doi.org/10.1016/j.orl.2013.08.007}

\bibitem{brooks2011handbook}
Brooks, S., Gelman, A., Jones, G., Meng, X.L.: Handbook of Markov Chain Monte
  Carlo. CRC press (2011)

\bibitem{ChopraRao1993}
Chopra, S., Rao, M.R.: {The partition problem}. Math. Prog.  \textbf{59}(1),
  87–115 (1993). \doi{10.1007/BF01581239}

\bibitem{repress}
Elowitz, M.B., Leibler, S.: {A synthetic oscillatory network of transcriptional
  regulators}. Nature  \textbf{403}(6767),  335--338 (2000).
  \doi{10.1038/35002125}

\bibitem{GleixnerEtal2018}
Gleixner, A., Bastubbe, M., Eifler, L., et. al: {The SCIP Optimization Suite
  6.0}. ZIB-Report 18-26 (2018),
  \url{https://nbn-resolving.org/urn:nbn:de:0297-zib-69361}

\bibitem{Grotschel1990}
Gr{\"o}tschel, M., Wakabayashi, Y.: {Facets of the clique partitioning
  polytope}. Math. Prog.  \textbf{47}(1),  367--387 (1990).
  \doi{10.1007/BF01580870}

\bibitem{hindrose}
Hindmarsh, J.L., Rose, R.: A model of neuronal bursting using three coupled
  first order differential equations. Proceedings of the Royal society of
  London B. Biological sciences  \textbf{221}(1222),  87--102 (1984).
  \doi{10.1098/rspb.1984.0024}

\bibitem{KernighanLin1970}
Kernighan, B.W., Lin, S.: An efficient heuristic procedure for partitioning
  graphs. Bell System Technical Journal  \textbf{49}(2),  291--307 (1970).
  \doi{10.1002/j.1538-7305.1970.tb01770.x}

\bibitem{Liberti2007}
Liberti, L.: {Compact linearization for binary quadratic problems}. 4OR
  \textbf{5}(3),  231--245 (2007). \doi{10.1007/s10288-006-0015-3}

\bibitem{Mallach2018}
Mallach, S.: Compact linearization for binary quadratic problems subject to
  assignment constraints. 4OR  \textbf{16}(3),  295--309 (2018).
  \doi{10.1007/s10288-017-0364-0}

\bibitem{mccormick1976computability}
McCormick, G.P.: Computability of global solutions to factorable nonconvex
  programs: Part {I} -- {Convex} underestimating problems. Math. Prog.
  \textbf{10}(1),  147--175 (1976). \doi{10.1007/BF01580665}

\bibitem{Padberg1989}
Padberg, M.: The boolean quadric polytope: Some characteristics, facets and
  relatives. Math. Prog.  \textbf{45}(1),  139--172 (1989).
  \doi{10.1007/BF01589101}

\bibitem{ReuterWeberFackeldeyetal.2018}
Reuter, B., Weber, M., Fackeldey, K., et~al.: Generalized {Markov} state
  modeling method for nonequilibrium biomolecular dynamics: Exemplified on
  {Amyloid} \textbeta \ conformational dynamics driven by an oscillating
  electric field. J. Chem. Theory Comput.  \textbf{14}(7),  3579--3594 (2018).
  \doi{10.1021/acs.jctc.8b00079}

\bibitem{WitzigEtal2018}
Witzig, J., Beckenbach, I., Eifler, L., et~al.: Mixed-integer programming for
  cycle detection in nonreversible {Markov} processes. Multiscale Modeling \&
  Simulation  \textbf{16}(1),  248--265 (2018). \doi{10.1137/16M1091162}

\end{thebibliography}
\end{document}